\documentclass[12pt,a4paper]{article}
\usepackage{amsmath,amssymb,amsbsy,amscd,amsthm}
\theoremstyle{plain}

\newtheorem{thm}{Theorem}[section]
\newtheorem{prop}[thm]{Proposition}
\newtheorem{lem}[thm]{Lemma}
\newtheorem{prob}[thm]{Problem}

\newcommand{\Gal }{\mathop{\rm Gal}\nolimits}
\newcommand{\Aut }{\mathop{\rm Aut}\nolimits}
\newcommand{\ch }{\mathop{\mathrm{char}}\nolimits}
\newcommand{\trd }{\mathop{\rm tr.deg}\nolimits}
\newcommand{\ord}{\mathop{\rm ord}\nolimits}

\newcommand{\ff}{{\bf f}}
\newcommand{\zs}{\{ 0\} }
\newcommand{\sm}{\setminus}

\newcommand{\id}{{\rm id}}

\newcommand{\nd}{\noindent}
\newcommand{\ol}{\overline}

\newcommand{\Q}{{\bf Q}}
\newcommand{\Z}{{\bf Z}}
\newcommand{\ep}{\epsilon}

\newcommand{\kx}{k[{\bf x}]}
\newcommand{\kxr}{k({\bf x})}
\newcommand{\kxz}{k[{\bf x},z]}
\newcommand{\kxzr}{k({\bf x},z)}

\newcommand{\A}{{\mathcal A}}

\begin{document}

\title{Hilbert's fourteenth problem and field modifications}

\date{}

\author{Shigeru Kuroda\thanks{
Partly supported by JSPS KAKENHI 
Grant Number 15K04826. }}

\footnotetext{
This paper is the final version 
of the preprint entitled 
``Hilbert's fourteenth problem and invariant fields 
of finite groups".}

\footnotetext{2010 {\it MSC}: 
Primary 13E15; Secondary 13A50, 14E07, 14E08}

\maketitle

\begin{abstract}
Let $\kxr =k(x_1,\ldots ,x_n)$ 
be the rational function field, 
and $k\subsetneqq L\subsetneqq \kxr $ 
an intermediate field. 
Then, 
{\it Hilbert's fourteenth problem} asks whether 
the $k$-algebra $A:=L\cap k[x_1,\ldots ,x_n]$ 
is finitely generated. 
Various counterexamples to this problem were already given, 
but the case $[\kxr :L]=2$ was open when $n=3$. 
In this paper, 
we study the problem in terms of 
the field-theoretic properties of $L$. 
We say that $L$ is {\it minimal} if 
the transcendence degree $r$ of $L$ 
over $k$ is equal to that of $A$. 
We show that, 
if $r\ge 2$ and $L$ is minimal, 
then there exists 
$\sigma \in \Aut _kk(x_1,\ldots ,x_{n+1})$ 
for which $\sigma (L(x_{n+1}))$ 
is minimal and a counterexample to the problem. 
Our result implies the existence of 
interesting new counterexamples 
including one 
with $n=3$ and $[\kxr :L]=2$. 
\end{abstract}

\section{Introduction and main results}
\setcounter{equation}{0}

Let $k$ be a field, 
$\kx =k[x_1,\ldots ,x_n]$ 
the polynomial ring in $n$ variables over $k$, 
and $\kxr :=Q(\kx )$, 
where $Q(R)$ denotes the field of fractions of 
$R$ for an integral domain $R$. 
In this paper, 
we give a simple and useful construction of 
counterexamples to the following problem.

\begin{prob}[Hilbert's fourteenth problem]
\label{prob:H14}
Let $k\subset L\subset \kxr $ be an intermediate field. 
Is the $k$-algebra $A:=L\cap \kx $ finitely generated?
\end{prob}

Since $\kx $ is normal, 
$A$ is integrally closed in $L$. 
This implies that 
$Q(A)$ is algebraically closed in $L$. 
We say that $L$ is {\it minimal} if 
the transcendence degree 
$r:=\trd _kA$ of $A$ over $k$ is equal to 
that of $L$, 
that is, $Q(A)=L$. 
Since $A=Q(A)\cap \kx $, 
we may assume that $L$ is minimal 
in Problem~\ref{prob:H14}.

By Zariski~\cite{Zariski}, 
the answer to Problem~\ref{prob:H14} 
is affirmative if $r\le 2$. 
If $r=1$, 
then $A=k[f]$ holds for some $f\in A$ (cf.~\cite{Zaks}), 
since $A$ is normal. 
Nagata~\cite{Nagata2} gave 
the first counterexample 
to Problem~\ref{prob:H14} 
when $(n,r)=(32,4)$. 
Later, 
Roberts~\cite{Rob} gave a different kind of 
counterexample when $(n,r)=(7,6)$ and $\ch k=0$. 
There are several generalizations of the results of 
Nagata 
(cf.\ 
\cite{Mukai1}, \cite{Steinberg} and \cite{Totaro}) 
and Roberts 
(cf.\ 
\cite{DF5}, \cite{F6}, \cite{KM} and \cite{Roberts}). 
In these counterexamples, 
$A$ are the invariant rings $\kx ^G$ 
for some subgroups $G$ of $\Aut _k\kx $. 
It is well known that
$\kx ^G$ is finitely generated if $G$ 
is a finite group (cf.~\cite{Noether}). 
Since $G\subset \Aut _{\kxr ^G}\kxr $, 
we know by Zariski~\cite{Zariski} that $\kx ^G$ is 
finitely generated if $n=3$ and $|G|=\infty $. 
There exist non-finitely generated 
invariant rings for $(n,r)=(5,4)$ 
(cf.\ \cite{DF5}), 
but no such example is known 
for $n=4$ or $r=3$.

Refining the method of 
\cite{Roberts}, 
the author gave counterexamples 
for $(n,r)=(4,3)$ in \cite{dim4}, 
and for $n=r=3$ in \cite{dim3}, 
not as invariant rings. 
When $n=3$, 
he also gave in \cite{algext} counterexamples 
with $[\kxr :L]=d$ for each $d\ge 3$. 
Some of them are the invariant fields 
for finite subgroups of $\Aut _k\kxr $. 
However, 
the case $d=2$ was not settled. 
This leads to the following more general problem.

In the following, 
let $k\subsetneqq M\subsetneqq \kxr $ 
be a minimal intermediate field. 
If $\phi :X\to Y$ is a map, 
then we sometimes write 
$Z^\phi :=\phi (Z)$ for $Z\subset X$.

\begin{prob}
\label{prob:FMP}
Assume that $\trd _kM\ge 3$. 
Does there always exist $\sigma \in \Aut _k\kxr $ 
such that $M^{\sigma }$ is minimal and a counterexample 
to Problem~$\ref{prob:H14}$? 
\end{prob}

We remark that Nagata~\cite{Nagata2} 
implies a positive answer to this problem 
when $n\ge 32$ and $M=k(x_1,\ldots ,x_4)$ 
(see also~\cite{Derksen}).

In this paper, 
we settle a `stable version' of Problem~\ref{prob:FMP} 
when $\ch k=0$. 
Namely, 
let $\kxz :=k[x_1,\ldots ,x_n,z]$ 
be the polynomial ring in $n+1$ variables over $k$, 
and let $\kxzr :=Q(\kxz )$. 
Then, 
we have the following theorem.

\begin{thm}\label{thm:SFMP}
Let $k$ be any field with $\ch k=0$. 
If $\trd _kM\ge 2$, $M\ne \kxr $ and $M$ is minimal, 
then there exists $\sigma \in \Aut _k\kxzr $ 
such that 
the $k$-algebra 
$\A :=M(z)^{\sigma }\cap \kxz $ is not finitely generated 
and $Q(\A )=M(z)^{\sigma }$. 
\end{thm}

Hence, 
there exists a counterexample $L$ with $[\kxr :L]=2$ 
for $n\ge 3$ (cf.\ \S 4.1). 
Theorem~\ref{thm:SFMP} 
also implies the existence of a counterexample $L$ 
which is not rational over $k$ 
(cf.\ \S 4.2). 
Such examples were previously not known.

Theorem~\ref{thm:SFMP} 
is a consequence of the following results. 
In the rest of this paper, 
$k$ denotes any field with $\ch k=0$. 
If $B$ is a $k$-domain, 
then we regard $\Aut _kB$ 
as a subgroup of $\Aut _kQ(B)$. 
We write $B_f:=B[1/f]$ for $f\in B\sm \zs $. 
The localization of $B$ 
at a prime ideal $\mathfrak{p}$ 
is denoted by $B_{\mathfrak{p}}$.

Now, 
assume that 
$M$ satisfies the following condition $(\dag )$:

\smallskip 

\nd ($\dag $) 
$R:=M\cap \kx =M\cap \kx _{x_1}$, 
$Q(R)=M$, 
and $\ol{R}:=R^{\ep }$ is not normal, 

\smallskip 

\nd where $\ep :\kx \to k[x_1]$ is the substitution map 
defined by $x_2,\ldots ,x_n\mapsto 0$. 
Since $Q(\ol{R})\cap k[x_1]$ is the integral closure of 
$\ol{R}$ in $Q(\ol{R})$, 
we can find 
$h\in Q(\ol{R})\cap k[x_1]$ 
not belonging to $\ol{R}$. 
Take $f,g\in R$ with 
$\ep (f)/\ep (g)=h$. 
Then, 
$\ep (g)$ is not in $k$, 
since $h$ is not in $\ol{R}$. 
Hence, 
$k[x_1]$ is integral over $k[\ep (g)]$. 
Thus, 
there exists a monic polynomial $\Pi (z)$ over $k[g]$ 
with $\ep (\Pi (f))=0$.

For $t=(t_i)_{i=2}^n\in \Z ^{n-1}$ 
and $h\in k[x_1]$ above, 
define 
$\theta _t^h\in \Aut _k\kxz _{x_1}$ by 
\begin{equation}\label{eq:theta}
\theta _t^h(x_1)=x_1^{-1},\ \ 
\theta _t^h(x_i)=x_1^{t_i}x_i\text{ for }i=2,\ldots ,n\ \ 
\text{and}\ \ 
\theta _t^h(z)=z+\theta _t^h(h). 
\end{equation}
We often write $\theta :=\theta _t^h$ for simplicity. 
Note that, 
if $p\in \ker \ep =\sum _{i=2}^nx_i\kx $ 
is of $x_1$-degree 
less than $t_2,\ldots ,t_n$, 
then $\theta (p)$ belongs to $x_1\kx $. 
Since $f-gh$ and $\Pi (f)$ 
lie in $\ker \ep $, 
we can find $t\in \Z ^{n-1}$ which satisfies

\smallskip

\nd ($\ddag $) 
$\theta (f-gh)\in \kx $ 
and $\theta (\pi )\in x_1\kx $, 
where $\pi :=\Pi (f)$.

\smallskip

In the notation above, 
the following theorems hold when $\ch k=0$.

\begin{thm}\label{thm:key}
Assume that $(\dag )$ and $(\ddag )$ are satisfied. 
Then, the $k$-algebra 
$\A :=M(z)^{\theta }\cap \kxz $ is not finitely generated 
and $Q(\A )=M(z)^{\theta }$. 

\end{thm}

Theorem~\ref{thm:key} implies Theorem~\ref{thm:SFMP} 
in the case where $M$ satisfies $(\dag )$. 
The general case is reduced to this case 
thanks to the following theorem. 

\begin{thm}\label{thm:FM}
Assume that 
$\trd _kM\ge 2$, $M\ne \kxr $ and $M$ is minimal. 
Then, there exists $\phi \in \Aut _k\kxr $ 
such that $M^{\phi }$ satisfies $(\dag )$ 
and $\kx ^{\phi }\subset \kx $. 
\end{thm}

We prove 
Theorems~\ref{thm:key} and \ref{thm:FM} 
in Sections~\ref{sect:key} 
and \ref{sect:FM}, respectively. 
The last section contains examples and remarks.

\section{Counterexamples}
\label{sect:key}

The goal of this section is to prove Theorem~\ref{thm:key}.

\begin{lem}\label{prop:intersect}
If $M$ satisfies $(\dag )$, 
then the following holds 
for any $t\in \Z ^{n-1}$.

\nd{\rm (i)} 
We have $\A =R[z]^{\theta }\cap \kxz $.

\nd{\rm (ii)} 
If $R[x]^{\theta }\cap x_1\kxz \ne \zs $, 
then $M(z)^{\theta }$ is minimal, 
and so $Q(\A )=M(z)^{\theta }$. 
\end{lem}

\begin{proof}
(i) 
It suffices to 
show that $B:=M(z)^{\theta }\cap \kxz _{x_1}$ 
is equal to $R[z]^{\theta }$, 
since $\A =B\cap \kxz $. 
We have 
$(\kxz _{x_1})^{\theta }=\kxz _{x_1}$, 
so $B=(M(z)\cap \kxz _{x_1})^{\theta }$. 
Also, 
$M(z)\cap \kxz _{x_1}
=M[z]\cap \kxz _{x_1}
=\bigl(M\cap \kx _{x_1}\bigr)[z]=R[z]$ 
holds by ($\dag $). 
Therefore, 
we obtain 
$B=(M(z)\cap \kxz _{x_1})^{\theta }=R[z]^{\theta }$.

\nd 
(ii) 
Take $0\ne q \in R[x]^{\theta }\cap x_1\kxz $. 
Then, 
for each $p\in R[z]^{\theta }$, 
there exists $l\ge 0$ with 
$pq^l,q^l\in R[z]^{\theta }\cap \kxz =\A $, 
and so $p\in Q(\A )$. 
Thus, 
$Q(\A )$ contains $R[z]^{\theta }$. 
Since $Q(R)=M$ by ($\dag $), 
we see that $\trd _k\A =\trd _kM(z)^{\theta }$. 
\end{proof}

Now, 
let us prove $Q(\A )=M(z)^{\theta }$. 
By virtue of Lemma~\ref{prop:intersect} (ii) and ($\ddag $), 
it suffices to verify that $\pi $ is nonzero. 
Suppose that $\pi =0$. 
Then, 
we have $\trd _kk(f,g)\le 1$. 
Hence, 
$k(f,g)\cap \kx =k[p]$ holds for some $p\in \kx $ 
(cf.~\S 1). 
Since $p$ lies in $M\cap \kx =R$, 
and $f$ and $g$ lie in $k[p]$, 
we have $\bar{p}:=\ep (p)\in \ol{R}$ 
and $h\in k(\bar{p})$. 
Because $\bar{p}$ and $h$ are elements of $k[x_1]$, 
we see that 
$h\in k(\bar{p})$ implies $h\in k[\bar{p}]$. 
Therefore, 
$h$ belongs to $\ol{R}$, 
a contradiction.

It remains to show that $\A $ is not finitely generated. 
Set $d:=\deg _z\Pi (z)$. 
Since $\theta (f)$ lies in $\kx _{x_1}$, 
we can find $e\ge 1$ 
such that $\theta (\pi )^e\theta (f)^i$ belongs to $\kx $ 
for $i=0,\ldots ,d-1$ by ($\ddag $). 
We extend $\ep $ to a substitution map 
$\kxz _{x_1}\to k[x_1,z]_{x_1}$ by $\ep (z)=z$. 
Our goal is to prove the following statements 
which imply that $\A $ is not finitely generated
(cf.~\cite[Lemma~2.1]{algext}):

\medskip

\nd (I) 
We have $\A ^{\ep }=k$. 
Hence, 
there do not exist $l\ge 1$ and $p\in \A $ 
for which the monomial $z^l$ appears in $p$. 

\smallskip

\nd (II) 
For each $l\ge 1$, 
there exists $q_l\in \A $ 
such that $\deg _z(q_l-\theta (\pi )^ez^l)<l$.

\medskip

\nd {\it Proof of} (I). 
Since $\theta (\ep (p))=\ep (\theta (p))$ 
holds for $p=x_1,\ldots ,x_n,z$, 
the same holds for all $p\in \kxz $. 
Suppose that $\A ^{\ep }\ne k$. 
Then, by Lemma~\ref{prop:intersect}~(i), 
we can find $p\in R[z]$ 
for which 
$\theta (p)\in \kxz $ 
and $\ep (\theta (p))\not\in k$. 
Note that 
$\theta (\ep (p))=\ep (\theta (p))$ lies in $k[x_1,z]$. 
Since $\ep (p)$ is in $\ol{R}[z]$, 
we also have 
$\theta (\ep (p))\in \ol{R}[z]^{\theta }\subset k[x_1^{-1},z]$. 
Thus, 
$\theta (\ep (p))$ lies in 
$T:=\ol{R}[z]^{\theta }\cap k[z]=\widetilde{R}[z+\tilde{h}]\cap k[z]$, 
where 
$\widetilde{R}:=
\theta (\ol{R})$ and 
$\tilde{h}:=\theta (h)$. 
Observe that 
$\partial q/\partial z$ is in $T$ 
whenever $q$ is in $T$. 
Since 
$\theta (\ep (p))$ is not in $k$, 
we see that 
$T$ contains a linear polynomial. 
Therefore, 
$\widetilde{R}[z+\tilde{h}]$ contains $z$. 
This implies $\tilde{h}\in \widetilde{R}$, 
and thus 
$h=\theta (\tilde{h})\in 
\theta (\widetilde{R})=\ol{R}$, 
a contradiction. 
\qed

\medskip

Since $\pi =\Pi (f)\in k[f,g]$ 
is monic of degree $d$ in $f$, 
we can write 
\begin{equation}\label{eq:decomp}
k\!\left[f,g,\frac{1}{g}\right]
=\sum _{i=0}^{d-1}f^i
k\!\left[\pi ,g,\frac{1}{g}\right]
=k[f,g]+N,
\text{ where }
N:=\sum _{i=0}^{d-1}f^i
k\!\left[\pi ,\frac{1}{g}\right]. 
\end{equation}
Since $\ep (g)$ is in $k[x_1]\sm k$ 
and $\theta (x_1)=x_1^{-1}$, 
we have $\theta (g)\in \kx _{x_1}\sm \kx $. 
Hence, 
$1/\theta (g)$ lies in the localization $\kx _{(x_1)}$. 
Since $\theta (\pi ^ef^i)\in \kx $ holds 
for $i=0,\ldots ,d-1$, 
we see that 
$\theta (\pi ^eN)$ is contained in $\kx _{(x_1)}$.

Now, let us fix $l\ge 0$. 
For $\ff =(f_j)_{j=1}^l\in k[f,g]^l$ 
and $i=0,\ldots ,l$, 
we define 
$$
P_i^{\ff }(z):=\pi ^e
\left( 
\frac{1}{i!}z^i+\frac{f_1}{(i-1)!}z^{i-1}
+\frac{f_2}{(i-2)!}z^{i-2}
+\cdots +f_i\right) \in k[f,g][z]. 
$$

\begin{lem}\label{lem:ff}
Set $r:=f/g$. 
Then, there exists $\ff (l)\in k[f,g]^l$ 
such that 
$\theta (P_0^{\ff (l)}(r)),\ldots ,
\theta (P_l^{\ff (l)}(r))$ 
belong to $\kx _{(x_1)}$. 
\end{lem}
\begin{proof}
We prove the lemma by induction on $l$. 
Since 
$\theta (\pi ^e)$ is in $\kx $, 
the case $l=0$ is true. 
Assume that $l\ge 1$ 
and $\ff (l-1)=(f_j)_{j=1}^{l-1}$ 
is defined. 
Set $p:=\sum _{j=1}^{l}(f_{l-j}/j!)r^j$, 
where $f_0:=1$. 
Then, 
$p$ belongs to 
$k[f,g,1/g]$. 
Hence, 
we can write $p=p'+p''$ by (\ref{eq:decomp}), 
where $p'\in k[f,g]$ and $p''\in N$. 
Set $f_l:=-p'$, 
and define $\ff (l):=(f_j)_{j=1}^l$. 
Then, 
we have $P_l^{\ff (l)}(r)=\pi ^e(p-p')
=\pi ^ep''\in \pi ^eN$. 
Since $\theta (\pi ^eN)\subset \kx _{(x_1)}$ 
as shown above, 
it follows that 
$\theta (P_l^{\ff (l)}(r))$ belongs to $\kx _{(x_1)}$. 
For $i=0,\ldots ,l-1$, we have 
$\theta (P_i^{\ff (l)}(r))
=\theta (P_i^{\ff (l-1)}(r))$ by definition, 
and $\theta (P_i^{\ff (l-1)}(r))$ 
belongs to $\kx _{(x_1)}$ by induction assumption. 
\end{proof}

\nd {\it Proof of} (II). 
Let $\ff :=\ff (l)$ 
be as in Lemma~\ref{lem:ff} 
and set $q:=\theta (P_l^{\ff }(z))$. 
Then, we have 
$\deg _z(l!\cdot q-\theta (\pi )^ez^l)<l$. 
So, 
we show that $q$ belongs to $\A =R[z]^{\theta }\cap \kxz $. 
Since 
$q\in R[z]^{\theta }\subset \kxz _{x_1}$ 
and $\kxz _{x_1}\cap \kx _{(x_1)}[z]=\kxz $, 
it suffices to check 
that $q$ lies in $\kx_{(x_1)}[z]$. 
By Taylor's formula, we have 
$$
q=\theta \left(
\sum _{i=0}^l\frac{1}{i!}
P_{l-i}^{\ff }(r)\cdot 
(z-r)^i
\right) 
=\sum _{i=0}^l\frac{1}{i!}
\theta (P_{l-i}^{\ff }(r))\cdot 
\left(z-\frac{\theta (f-gh)}{\theta (g)}\right)^i, 
$$
where $r:=f/g$. 
Since 
$\theta (P_0^{\ff }(r)),\ldots ,
\theta (P_l^{\ff }(r))$ 
and $1/\theta (g)$ 
are in $\kx _{(x_1)}$, 
and $\theta (f-gh)$ is in $\kx $ by ($\ddag $), 
we see that $q$ belongs to $\kx_{(x_1)}[z]$. 
\qed

\medskip 

This completes the proof of Theorem~\ref{thm:key}.

\section{Field modification}\label{sect:FM}

In this section, 
we prove Theorem~\ref{thm:FM}. 
Set $R:=M\cap \kx $. 
Then, 
we have $Q(R)=M\ne \kxr $ and $r:=\trd _kR=\trd _kM\ge 2$ 
by assumption.

\begin{lem}\label{lem:FM}
If $x_1$ is transcendental over $M$, 
or if $r=n$ and $x_1\not\in M$, 
then 
$M\cap \kx _{x_1-\alpha }=M\cap \kx $ holds 
for all but finitely many $\alpha \in k$. 
\end{lem}
\begin{proof}
First, 
assume that 
$x_1$ is transcendental over $M$. 
Let $y_1,\ldots ,y_r\in \kxr $ 
be a transcendence basis of $\kxr $ over $M$ 
with $y_1=x_1$. 
Set $S:=M[y_1,\ldots ,y_r]$. 
Then, there exists $u\in S\sm \zs $ such that 
$T:=S_u[x_1,\ldots ,x_n]$ 
is integral over $S_u$. 
Note that 
$p:=y_1-\alpha $ is a prime in $S_u$, 
i.e., $u\not\in pS$, 
for all but finitely many $\alpha \in k$. 
Such a $p$ is not a unit of $T$, 
because a prime ideal of $T$ lies over $pS_u$. 
Hence, 
we have $pT\cap T^*=\emptyset $. 
Since $\kx \subset T$ and 
$R\sm \zs \subset M^*\subset T^*$, 
we get 
$p\kx \cap R\subset pT\cap R=\zs $. 
Therefore, noting $M=Q(R)$, 
we see that $M\cap \kx _p=M\cap \kx $.

Next, 
assume that $r=n$ and $x_1\not\in M$. 
Let $f(z)$ be the minimal polynomial of $x_1$ over $M$. 
Since $x_1,\ldots ,x_n$ 
are algebraic over $M=Q(R)$, 
there exists $u\in R\sm \zs $ for which 
$B:=\kx_u$ is integral over $A:=R_u$. 
We can choose $u$ so that $f(z)$ 
lies in $A[z]$, 
and the discriminant $\delta $ of $f(z)$ 
is a unit of $A$. 
Note that 
$B$ is integral over a finitely 
generated 
$k$-subalgebra $A'$ of $A$. 
Since $B$ is a finite $A'$-module, 
so is the $A'$-submodule $A$. 
Thus, 
$A$ is Noetherian. 
We also note that $A$ is normal, 
$B$ is factorial, 
and $g(z):=f(z)/(x_1-z)\in B[z]\sm B$. 
As before, 
$p:=x_1-\alpha $ is a prime in $B$, i.e., 
$u\not\in p\kx $,  
for all but finitely many $\alpha \in k$. 
Take such a $p$, 
and set $\mathfrak{p}:=A\cap pB$. 
Then, 
we have $B_p\cap A_{\mathfrak{p}}\subset B$. 
Since $\kx _p\cap B=\kx $, 
it follows that 
$\kx _p\cap A_{\mathfrak{p}}\subset \kx $. 
Our goal is to show that 
$R':=M\cap \kx _p$ is contained in $A_{\mathfrak{p}}$, 
which implies that $R'=M\cap \kx $.

Since $\delta \in A^*\subset B^*$, 
the image of $f(z)$ in 
$(B/pB)[z]$ has no multiple roots. 
Since $f(\alpha )=pg(\alpha )$, 
it follows that 
$g(\alpha )\not\in pB$. 
We show that $g(\alpha )\not\in B^*$. 
Let $E$ be a Galois closure of $\kxr $ over $M$, 
and $C$ the integral closure of $B$ in $E$. 
Then, $p$ is not in $C^*$ as before, 
and $C^{\sigma }=C$ 
for each $\sigma \in \mathcal{G}:=\Gal (E/M)$. 
Hence, 
$\sigma (p)\not\in C^*$ 
holds for all $\sigma \in \mathcal{G}$. 
Since $p$ is a root of $h(z):=f(z+\alpha )$, 
the other roots $p_2,\ldots ,p_l$ of $h(z)$ lie in $C\sm C^*$. 
Thus, 
we see from 
the relation $(-1)^lpp_2\cdots p_l=h(0)=pg(\alpha )$ 
that $g(\alpha )$ is not in $C^*$, 
and hence in $B^*$. 
Since $B$ is factorial, 
there exists a prime $q\in B$ 
satisfying $g(\alpha )\in qB$. 
Since 
$g(\alpha )$ is not in $pB$, 
we know that $qB\ne pB$. 
Hence, 
we have $\kx _p\subset B_p\subset B_{(q)}$. 
Therefore, 
$R'$ is contained in $A_1:=M\cap B_{(q)}$.

Since $B$ is integral over $A$, 
and $A$ is normal, 
the prime ideals 
$\mathfrak{p}$ and $\mathfrak{q}:=A\cap qB$ 
of $A$ are of height one. 
Since $A$ is Noetherian, 
$A_{\mathfrak{q}}$ 
is a discrete valuation ring. 
Note also that 
$A_{\mathfrak{q}}\subset A_1\subset M=Q(A_{\mathfrak{q}})$. 
Hence, 
$A_1$ is the localization of $A_{\mathfrak{q}}$ 
at a prime ideal (cf.~\cite[\S Thm.\ 10.1]{Matsumura}). 
Since $A_{\mathfrak{q}}$ is a discrete valuation ring, 
$A_1$ must be $A_{\mathfrak{q}}$ or $M$. 
We claim that $A_1$ is not a field, 
since $A_1$ contains $\mathfrak{q}\ne 0$, 
and 
$\mathfrak{q}\cap A_1^*\subset 
qB\cap B_{(q)}^*=\emptyset $. 
Thus, 
$A_1$ equals $A_{\mathfrak{q}}$. 
Since $R'\subset A_1$ as mentioned, 
we obtain that $R'\subset A_{\mathfrak{q}}$.

Finally, 
we show that $\mathfrak{p}=\mathfrak{q}$. 
Let $N:=N_{E/M}:E\to M$ be the norm function. 
Then, for each $c\in C$, 
we have $N(c)\in C\cap M=C\cap Q(A)$. 
Since $A$ is normal, 
this implies $N(c)\in A$. 
Now, we prove 
$\mathfrak{p}\subset \mathfrak{q}$. 
Take $b\in B$ with $pb\in A$. 
Then, 
we have 
$(pb)^{[E:M]}=N(pb)=N(p)N(b)$. 
Since $N(p)$ is a power of $\pm h(0)$, 
and $h(0)=pg(\alpha )$ with $g(\alpha )\in qB$, 
it follows that $pb\in \mathfrak{q}$, 
proving $\mathfrak{p}\subset \mathfrak{q}$. 
Since $\mathfrak{p}$ and $\mathfrak{q}$ 
have the same height, 
this implies that $\mathfrak{p}=\mathfrak{q}$. 
\end{proof}

The following remarks are used 
in the proof of Theorem~\ref{thm:FM}:

\nd 
(a) Let $\phi \in \Aut _k\kxr $. 
If $\kx ^{\phi }\subset \kx $, 
then we have $R^{\phi }\subset M^{\phi }\cap \kx $. 
Since $Q(R)=M$ by assumption, 
this implies that 
$M^{\phi }\subset Q(M^{\phi }\cap \kx )$. 
Hence, 
$M^{\phi }$ is minimal. 
If moreover 
$M^{\phi }\cap \kx _{x_1}=R^{\phi }$, 
then we have $M^{\phi }\cap \kx =R^{\phi }$.

\nd (b) 
Let $A$ be a normal $k$-subalgebra of $k[x_1]$. 
Then, 
$A=k[p]$ holds for some $p\in A$ 
(cf.~\cite{Zaks}). 
Hence, 
the additive semigroup 
$\Sigma _A:=\{ \ord f\mid f\in A\sm \zs \} $ 
is single-generated, 
where 
$\ord f:=\max \{ i\in \Z \mid f\in x_1^ik[x_1]\} $.

\medskip 

\nd {\it Proof of Theorem} \ref{thm:FM}. 
By (a), 
it suffices to find $\phi \in \Aut _k\kxr $ 
for which 
$\kx ^{\phi }\subset \kx $, 
$M^{\phi }\cap \kx _{x_1}=R^{\phi }$, 
and $(R^{\phi })^{\ep }$ is not normal. 
Take $1\le i_0\le n$ such that 
$x_{i_0}$ is transcendental over $M$ if $r<n$, 
and $x_{i_0}\not\in M$ if $r=n$. 
By replacing $x_i$ with $x_i+x_{i_0}$ 
if necessary for $i\ne i_0$, 
we may assume that 
$x_1,\ldots ,x_n$ are transcendental over $M$ 
if $r<n$, 
and $x_1,\ldots ,x_n\not\in M$ 
if $r=n$. 
Let $V\subset \sum _{i=1}^nkx_i$ be the $k$-vector space 
generated by $f^{\rm lin}$ for $f\in R$, 
where $f^{\rm lin}$ is the linear part of $f$. 
Then, 
$\dim _kV$ is at most $r$. 
Let $f_1,\ldots ,f_r\in R$ 
be algebraically independent over $k$. 
Then, the $r\times n$ matrix 
$[\partial f_i/\partial x_j]_{i,j}$ is of rank $r$. 
Hence, 
there exists a 
Zariski open subset $\emptyset \ne U\subset k^n$ 
such that, 
for each $a\in U$, 
the rank of 
$[(\partial f_i/\partial x_j)(a)]_{i,j}$ 
is $r$. 
Define $\tau _{a}\in \Aut _k\kx $ by 
$\tau _a(x_i)=x_i+a_i$ for each $i$, 
where $a=(a_1,\ldots ,a_n)$. 
Then, 
$\tau _{a}(f_i)^{\rm lin}$ is written as 
$\sum _{j=1}^n(\partial f_i/\partial x_j)(a)\cdot x_j$ 
for each $i$. 
Thus, 
replacing $M$ with $M^{\tau _{a}}$ 
for a suitable $a\in U$, 
we may assume that $\dim _kV=r$, 
and also $M\cap \kx _{x_i}=R$ for all $i$ 
by Lemma~\ref{lem:FM}. 
Since $r\ge 2$ by assumption, 
we may change the indices of $x_1,\ldots ,x_n$ 
so that, for $i=1,2$, there exists $g_i\in R$ with 
$g_i^{\rm lin}\in x_i+\sum _{j=3}^nkx_j$.

Now, 
define $\rho \in \Aut _{k[x_2,\ldots ,x_n]_{x_2}}\kx _{x_2}$ by 
$\rho (x_1)=x_1x_2$. 
Then, 
we have $\kx ^{\rho }\subset \kx $. 
Since $M\cap \kx _{x_2}=R$, 
we also have 
$M^{\rho }\cap \kx _{x_2}=R^{\rho }$. 
Thus, 
$M^{\rho }$ is minimal and 
$M^{\rho }\cap \kx =R^{\rho }$ by (a). 
If $r<n$, 
there exists 
$\alpha \in k^*$ 
such that $y:=x_1+\alpha x_2$ 
is transcendental over $M^{\rho }$, 
since so is $x_2=\rho (x_2)$. 
Similarly, 
if $r=n$, 
then 
$y:=x_1+\alpha x_2\not\in M^{\rho }$ holds 
for some $\alpha \in k^*$. 
In either case, 
there exists $\beta \in k$ 
such that $M^{\rho }\cap \kx _{y-\beta }=R^{\rho }$ 
by Lemma~\ref{lem:FM}, 
since 
$M^{\rho }$ is minimal, 
$M^{\rho }\cap \kx =R^{\rho }$ and 
$y,x_2,\ldots ,x_n$ 
is a system of variables. 
Define $\psi \in \Aut _k\kx $ by 
$\psi (y)=x_1+\beta $, 
$\psi (x_2)=x_2+x_1^2$ 
and $\psi (x_i)=x_i$ for $i\ne 1,2$, 
and set $\phi :=\psi \circ \rho $. 
Then, 
we have $\kx ^{\phi }\subset \kx $, 
and $M^{\phi }\cap \kx _{x_1}=R^{\phi }$, 
since $M^{\rho }\cap \kx _{y-\beta }=R^{\rho }$. 
Note that 
$\ep (\phi (x_i))=0$ for $i\ne 1,2$, 
$\ep (\phi (x_2))=x_1^2$ 
and 
$\ep (\phi (x_1))
\in (x_1+\beta )x_1^2+x_1^4k[x_1]$. 
Hence, 
we have 
$(R^{\phi })^{\ep }\subset k+x_1^2k[x_1]$, 
and 
$\ep (\phi (g_i))$ lies in 
$k+\ep (\phi (x_i))+x_1^4k[x_1]$ for $i=1,2$. 
Therefore, 
we get $1\not\in \Sigma _{(R^{\phi })^{\ep }}$ 
and $2,3\in \Sigma _{(R^{\phi })^{\ep }}$. 
By (b), 
this implies that $(R^{\phi })^{\ep }$ is not normal. 
\qed

\section{Examples and remarks}\label{sect:exrem}

\nd {\bf 4.1.} 
Define a system $y_1,\ldots ,y_n$ of variables by 
$y_2:=x_2-x_1+x_1^2$ and $y_i:=x_i$ for $i\ne 2$. 
Let $G$ be a permutation group on $\{ y_1,\ldots ,y_n\} $ 
such that $\tau (y_1)=y_2$ for some $\tau \in G$. 
We regard $G\subset \Aut _k\kx $ in a natural way. 
Then, we have

\begin{prop}\label{prop:example}
$\kxr ^G$ satisfies the condition $(\dag )$ 
with $\ol{R}=k[x_1^2,x_1^3]$. 
\end{prop}
\begin{proof}
Since $\tau (\kxr ^G)=\kxr ^G$ and $\tau (x_1)=y_2$, 
we have 
$$
\kxr ^G\cap \kx _{x_1}=
\kxr ^G\cap \kx _{x_1}\cap \kx _{y_2}=
\kxr ^G\cap \kx =R. 
$$
Since 
$R$ contains symmetric polynomials in $y_1,\ldots ,y_n$, 
we have $\trd _kR=n$. 
Hence, $\kxr ^G$ is minimal. 
The $k$-vector space $R$ 
is generated by $I_m$ for the monomials $m$ 
in $y_1,\ldots ,y_n$, 
where $I_m$ is the sum of the elements of the $G$-orbit of $m$. 
Since $\ep (I_{y_1})=x_1^2$, $\ep (I_{y_1y_2})=x_1^3-x_1^2$ 
and $\ep (I_m)\in x_1^2k[x_1]$ for all monomials $m\ne 1$, 
we have $\ol{R}=k[x_1^2,x_1^3]$, 
which is not normal. 
\end{proof}

For example, 
assume that $n=2$ and let $G=\langle \tau \rangle $. 
Then, 
we have $\kxr ^G=k(y_1+y_2,y_1y_2)$. 
Define $\theta =\theta _{(t_2)}^h\in \Aut _k\kxz _{x_1}$ 
as in (\ref{eq:theta}) 
for $h:=x_1$ and $t_2\ge 5$. 
Set $f:=y_1+y_2+y_1y_2$ and $g:=y_1+y_2$. 
Then, 
we have $\ep (f)/\ep (g)=h$, 
$\ep (f^2-g^3)=0$, 
and $\theta (f-gh),\theta (f^2-g^3)\in x_1\kx $. 
Therefore, 
$$
L:=k(y_1+y_2,y_1y_2,z)^{\theta }
=k(x_1^{t_2}x_2+x_1^{-2},
x_1^{t_2-1}x_2-x_1^{-2}+x_1^{-3}, 
z+x_1^{-1})
$$
is a counterexample to Problem~\ref{prob:H14}. 
We note that $[\kxzr :L]=2$.

\smallskip

\nd {\bf 4.2.} 
Let $G$ be a finite group with $|G|=n$, 
and write 
$\kxr =k(\{ x_{\sigma }\mid \sigma \in G\} )$. 
Let $k(G)$ be the invariant subfield of $\kxr $ 
for the $G$-action defined by 
$\tau \cdot x_{\sigma }:=x_{\tau \sigma }$ 
for each $\tau ,\sigma \in G$. 
Then, 
{\it Noether's 
Problem} asks whether 
$k(G)$ is {\it rational} over $k$, 
i.e., 
a purely transcendental extension of $k$. 
For various primes $p$, say $p=47$, 
it is known that $\Q (\Z /p\Z )$ is not rational over $\Q $ 
(cf.~\cite{Swan}). 
When $G$ is a finite abelian group, 
it is also known that 
$\Q (G)$ is rational over $\Q $ 
if and only if 
$\Q (G)(z_1,\ldots ,z_l)$ 
is rational over $\Q $ 
for some variables $z_1,\ldots ,z_l$ 
(cf.~\cite{EM}). 
These results and Theorem~\ref{thm:SFMP} 
imply the existence of 
a non-rational, minimal 
counterexample to Problem~\ref{prob:H14} 
for $k=\Q $ and $n\ge 48$. 
Since $G$ is considered as a permutation group on 
$\{ x_{\sigma }\mid \sigma \in G\} $, 
we can explicitly construct such an example 
using Proposition~\ref{prop:example} 
and Theorem~\ref{thm:key}.

\smallskip

\nd {\bf 4.3.} 
Let $B$ be a $k$-domain, 
$F:=Q(B)$, 
and $D\ne 0$ a {\it locally nilpotent $k$-derivation} of $B$, 
i.e., 
a $k$-derivation of $B$ such that, 
for each $a\in B$, 
there exists $l>0$ satisfying $D^l(a)=0$. 
Then, 
$\ker D$ is a $k$-subalgebra of $B$. 
We remark that 
there exists $\iota \in \Aut _kF$ for which 
$\iota ^2=\id $ and 
$F^{\langle \iota \rangle }\cap B=\ker D$ 
for the following reason: 
We can find $s\in B$ with 
$D(s)\ne 0$ and $D^2(s)=0$. 
It is known that such an $s$, 
called a {\it preslice} of $D$, 
is transcendental over $K:=Q(\ker D)$,  
and $F=K(s)$ and $B\subset K[s]$ hold 
(cf.\ e.g.~\cite[\S 1.3]{Essen}). 
Now, define $\iota \in \Aut _KK(s)\subset \Aut _kF$ 
by $\iota (s)=s^{-1}$. 
Then, 
we have $\iota ^2=\id $ and 
\begin{equation}\label{eq:involution}
F^{\langle \iota \rangle }\cap B
=F^{\langle \iota \rangle }\cap K[s]\cap B
=K(s+s^{-1})\cap K[s]\cap B
=K\cap B
=\ker D. 
\end{equation}

Assume that $B=\kx $, 
and the $k$-algebra $\ker D$ is not finitely generated 
(see \cite{DF5}, \cite{F6}, \cite{KM}, \cite{Roberts} 
and \cite{Rob} for such examples). 
Then, 
by (\ref{eq:involution}), 
we obtain counterexamples to Problem~\ref{prob:H14} 
of the form not only $L=Q(\ker D)$, 
but also $L=\kxr ^{\langle \iota \rangle }$. 
However, 
$\kxr ^{\langle \iota \rangle }$ 
is not minimal, 
since $\trd _k(\ker D)=n-1$.

\bibliographystyle{amsplain}

\noindent
Department of Mathematics and Information Sciences\\ 
Tokyo Metropolitan University \\ 
1-1  Minami-Osawa, Hachioji, 
Tokyo 192-0397, Japan\\
kuroda@tmu.ac.jp

\end{document}